\documentclass[12pt,epsfig,reqno]{amsart}
\usepackage{latexsym}
\usepackage{amsmath,amscd,amssymb,thmtools,amsfonts,mathrsfs}
\usepackage{graphicx}
\usepackage{epstopdf}
 2

\usepackage[utf8]{inputenc}
\usepackage[T1]{fontenc}

\include{plaquettemacro}
\usepackage{hyperref}
\author[Grover]{Priyanka Grover}
\author[Mishra]{Pradip Mishra}
\email{priyanka.grover@snu.edu.in, pradip.kumar@snu.edu.in}
\title{Perturbation bounds for the Mostow and the bipolar decompositions}
\address{Department of Mathematics\\
Shiv Nadar University, Dadri\\ U.P. 201314, India.}

\newcommand{\C}{\mathbb{C}}
\newcommand{\R}{\mathbb{R}}

\newcommand{\De}{{\rm D}}
\newcommand{\diag}{{\rm diag}}
\newtheorem{thm}{Theorem}

\newtheorem{corollary}[thm]{Corollary}

\newtheorem{prop}[thm]{Proposition}

\theoremstyle{remark}
\newtheorem{remark}[thm]{Remark}

\newcommand{\h}{\mathop{{\rm Re}}}

\newcommand{\vertiii}[1]{{\left\vert\kern-0.25ex\left\vert\kern-0.25ex\left\vert #1
    \right\vert\kern-0.25ex\right\vert\kern-0.25ex\right\vert}}
\begin{document}
\maketitle
\begin{abstract}
 Perturbation bounds for Mostow's decomposition and the bipolar decomposition of matrices have been computed. To do so, expressions for the derivative of the geometric mean of two positive definite matrices have been derived.
\end{abstract}

\numberwithin{thm}{section}
\numberwithin{remark}{section}
\numberwithin{prop}{section}
\numberwithin{corollary}{section}
\numberwithin{equation}{section}

\section{Introduction}
Matrix factorizations have been used in numerical analysis to implement efficient matrix algorithms. In machine learning, matrix factorizations play an important role to explain latent features underlying the interactions between different kinds of entities. Many matrix factorizations namely, the polar decomposition, the $QR$ decomposition, the $LR$ decomposition etc., have been of considerable interest for many decades. Perturbation bounds for such factorizations have been of interest for a long time (see \cite{bhatia_Matrix_factorizations, stewart, sun_stewart} and the references therein). Some generalizations and improvements on them have been obtained in the subsequent works, for example, see \cite{chang,chang_stehle,galantai,largillier,li_sun,li_sun2,li_yang_shao,xie}.

An interesting matrix factorization follows from the work of Mostow \cite{mostow}. It states  that every non singular complex matrix $Z$ can be uniquely factorized as
\begin{equation}
Z=We^{iK} e^S, \label{mostow}
\end{equation}
where $W$ is a unitary matrix, $S$ is a real symmetric matrix and $K$ is a real skew symmetric matrix. Recently, Bhatia \cite{bhatia_bipolar} showed that every complex unitary matrix $W$  can be factorized as
\begin{equation}
W=e^L e^{iT} \label{unitary},
\end{equation}where $L$ is a real skew symmetric matrix and $T$ is a real symmetric matrix.
Using \eqref{mostow} and \eqref{unitary}, it has been obtained in \cite{bhatia_bipolar} that
\begin{equation}
Z=e^L e^{iT} e^{iK} e^S. \label{bipolar}
\end{equation}
Our goal is to find the perturbation bounds for the factors arising in \eqref{mostow}, \eqref{unitary} and \eqref{bipolar}.  In \cite{barbaresco}, Barbaresco has used Berger Fibration in Unit Siegel Disk for Radar Space-Time Adaptive Processing and Toeplitz-Block-Toeplitz covariance matrices based on  Mostow's decomposition.

Let $\mathbb M(n, \C)$ be the space of $n\times n$ complex matrices,
and $\mathbb U(n, \C)$ be the set of $n\times n$ complex unitary matrices.
Let $|||\cdot|||$ be any unitarily invariant norm on $\mathbb M(n,\C)$, that is, for any $U,V\in \mathbb U(n,\C)$ and $A\in \mathbb M(n,\C)$, we have
$$|||UAV|||=|||A|||.$$ Two special examples of such norms are the \emph{operator norm} $\|\cdot\|$ (also known as the \emph{spectral norm}) and \emph{Frobenius norm} $\|\cdot\|_2$ (also known as \emph{Hilbert-Schmidt norm} or \emph{Schatten $2$-norm}).
Various properties of unitarily invariant norms are known \cite[Chapter IV]{bhatia_matrix_analysis}. We would require the following important properties:
for $A,B,C\in \mathbb M(n,\C)$
\begin{equation}
|||ABC|||\leq \|A\|\ |||B|||\ \|C\|,\label{tineq}
\end{equation}
and
\begin{equation}
|||A|||=|||A^*|||=|||A^t|||=|||\bar A|||.
\end{equation}
Let $\mathscr W$ be a subspace of $\mathbb M(n,\C)$ and let $\mathcal T:\mathscr W\rightarrow \mathbb M(n,\C)$ be a linear map. As in \cite{bhatia_Matrix_factorizations}, we take
\begin{eqnarray}
|||\mathcal T|||&=&\sup\{|||\mathcal T(X)|||: |||X|||=1 \}.\label{opnorm}
\end{eqnarray}

 It has been shown in \cite{barbaresco, bhatia_bipolar} that the factors in the decomposition \eqref{mostow} are related to the \emph{geometric mean}. So to obtain the perturbation bounds for \eqref{mostow},  we obtain expressions for the derivative of the geometric mean and bounds on its norms in Section 2. In Section 3 and Section 4, we exploit the idea in \cite{bhatia_Matrix_factorizations} to obtain bounds on the derivative of the decomposition maps for \eqref{mostow} and \eqref{unitary}, respectively. In Section 5, we discuss  the first order perturbation bounds for maps on Lie groups and obtain the perturbation bounds for the factorizations \eqref{mostow}, \eqref{unitary} and \eqref{bipolar}.

\section{Derivative of the geometric mean}
Let $\mathbb H(n,\C)$ be the space of $n\times n$ complex Hermitian matrices and  let $\mathbb P(n,\C)$ be the set of $n\times n$ complex positive definite matrices.
For $A,B\in \mathbb P(n,\C)$ their  geometric mean is defined as
\begin{equation}
A\#B=A^{1/2} \left(A^{-1/2} B A^{-1/2}\right)^{1/2} A^{1/2},\label{defnofgm}
\end{equation}
\cite[Chapter 4]{bhatia_positive_definite}. It is the unique positive solution of the \emph{Riccati equation}
 \begin{equation}
XA^{-1} X=B.\label{equationofgm}
\end{equation}
The geometric mean of $A$ and $B$ is also given by
\begin{equation}
A\#B=A (A^{-1} B)^{1/2}=(AB^{-1})^{1/2} B,\label{gm}
\end{equation}
where $(A^{-1}B)^{1/2}$ and $(AB^{-1})^{1/2} $ are the unique square roots of $A^{-1}B$ and $AB^{-1}$, respectively, with positive eigenvalues.

Let $G: \mathbb P(n, \C) \times \mathbb P(n, \C)\rightarrow \mathbb P(n,\C)$ be the map defined as
$$G(A,B)=A\# B.$$ Since $A\mapsto A^{1/2}$ is a differentiable function on $\mathbb P(n,\C)$, 
we get from \eqref{defnofgm} that $G$ is a differentiable map. The derivative  is given by
$$\De G(A,B)(X,Y)=\left.\frac{d}{dt}\right|_{t=0} G(A+tX, B+tY) \text{ for all } X,Y\in \mathbb H(n,\C).$$
The following proposition gives an expression for $\De G(A,B)$.

\begin{prop}\label{prop2.1}
For $X,Y\in \mathbb H(n,\C)$
\begin{equation}
\De G(A,B)(X,Y)=\int_0^{\infty} e^{-t (BA^{-1})^{1/2}} (Y+(BA^{-1})^{1/2} X (A^{-1} B)^{1/2}) e^{-t(A^{-1} B)^{1/2}} dt. \label{2.4}
\end{equation}
\end{prop}
\begin{proof}
 For sufficiently small $t$, by \eqref{equationofgm}, we have
\begin{equation}
G(A+tX,B+tY) (A+tX)^{-1} G(A+tX,B+tY)=B+tY.\label{2.1}
\end{equation}
Differentiating with respect to $t$  at $0$, we get
\begin{eqnarray*}
&&\hspace{-1.3cm}\left(\De G(A,B) (X,Y)\right) A^{-1} G(A,B)-G(A,B)(A^{-1}XA^{-1})G(A,B)\\
&&+G(A,B) A^{-1} \left(\De G(A,B)(X,Y)\right)=Y.
\end{eqnarray*}
Put $D=\De G(A,B)(X,Y)$ and $C=A^{-1} G(A,B)=(A^{-1} B)^{1/2}$. Then the above equation can be rewritten as
\begin{equation}
C^* D+D C=Y+C^* X C.\label{sylvester_gm}
\end{equation}
This is a well studied Sylvester's equation (see \cite{bhatia_matrix_analysis, bhatia_rosenthal}). By \cite[Theorem VII.2.3]{bhatia_matrix_analysis}, we obtain
\begin{equation}
\De G(A,B)(X,Y)=\int_0^{\infty} e^{-t C^*} (Y+C^*XC)   e^{-t C} dt. \label{eq2.7}
\end{equation}
Substituting $C=(A^{-1}B)^{1/2}$ in \eqref{eq2.7}, we obtain \eqref{2.4}.
\end{proof}

Some other expressions for the solution of the Sylvester's equation \cite{ bhatia_matrix_analysis, bhatiauchiyama} are known. From these, one can obtain other expressions for $\De G(A,B)(X,Y)$.

Suppose $A$ and $B$ commute. Then $C=(A^{-1}B)^{1/2}$ is Hermitian.  Let $\lambda_1(C)\geq \cdots\geq \lambda_n(C)$ denote the eigenvalues of $C$. Using \cite[Theorem VII.2.15]{bhatia_matrix_analysis} for \eqref{sylvester_gm}, we obtain

$$|||\De G(A,B)(X,Y)|||\leq \frac{\pi}{4 \lambda_n(C)} |||Y+C^*XC|||.$$
By \eqref{opnorm}, $|||\De G(A,B)|||=\sup\{|||\mathcal \De G(A,B)(X,Y)|||: |||(X,Y)|||=1\}$, where $|||(X,Y)|||= \max \{|||X|||,|||Y|||\}$. So we get
\begin{equation}
|||\De G(A,B)|||\leq \frac{\pi}{4 \lambda_n(C)} \left(1+\|C\|^2\right).\label{eq2.8}
\end{equation}

By Proposition \ref{prop2.1}, we obtain a better bound for $|||\De G(A,B)|||$. We mention this in the following corollary for general $A$ and $B$.
\begin{corollary}
For $A,B\in \mathbb P(n,\mathbb C)$
\begin{equation}
|||\De G(A,B)|||\leq
\left(\int_0^{\infty} \|e^{-t (A^{-1} B)^{1/2}}\|^2 dt\right)
\left(1+\|(A^{-1} B)^{1/2}\|^2\right).\label{bound_gm}
\end{equation}
\end{corollary}

In the case when $A$ and $B$ commute, $\int_0^{\infty} \|e^{-t C}\|^2 dt=\frac{1}{2\lambda_n(C)}.$
So from \eqref{bound_gm}, we obtain
\begin{equation}
|||\De G(A,B)|||\leq \frac{1}{2 \lambda_n(C)} \left(1+\|C\|^2\right).
\end{equation}
In some other cases, a bound on $\int_0^{\infty} \|e^{-t C}\|^2 dt$   is easy to calculate. For example, if  $\lambda_n\left(\h C\right)$ is nonnegative (where $\h C=\frac{C+C^*}{2}$), then we have

$$\int_0^{\infty} \|e^{-t C}\|^2 dt\leq \frac{1}{2 \lambda_n\left(\h C\right)}.$$
This has been observed in \cite{bhatia_Matrix_factorizations, bhatia_elsner}.

\begin{remark}
We observe that for $A,B\in \mathbb P(n,\mathbb C)$, $\De G(A,B)$ is a positive linear map from $\mathbb H(n, \C)\times \mathbb H(n, \C)$ to $\mathbb M(n,\C)$. So by \cite[Theorem 2.6.3]{bhatia_positive_definite}, we obtain $$\|\De G(A,B)\|= \|\De G(A,B)(I,I)\|. $$
\end{remark}

\section{Mostow's decomposition}
The Mostow decomposition theorem \eqref{mostow} gives that any non singular matrix $Z$ can be uniquely factorized as $Z=W e^{iK} e^S$. Let $P_1=e^{iK}$ and $P_2=e^S$.  Then $P_1\in \mathbb P(n, \C)$ and $P_2 \in \mathbb P(n,\R)$, where $\mathbb P(n,\R)$ stands for the set of $n\times n$ real positive definite matrices. We also have $\overline {P_1} P_1=I$. Such matrices $X$  which satisfy $\overline X X=I$ are  called \emph{circular} (or \emph{coninvolutary}) \cite{horn_tma}. Let $\mathbb P_{circ}$ be the set of circular positive definite matrices. Then $P_1\in \mathbb P_{circ}$.
Let $\varrho: \mathbb{GL}(n,\C)\rightarrow \mathbb{U}(n,\C)\times \mathbb P_{circ}\times \mathbb P(n,\R)$ be the map
\begin{equation}
\varrho( Z)= (\varrho_0(Z), \varrho_1(Z), \varrho_2(Z)),
\end{equation}
where $\varrho_0(Z)=W$, $\varrho_1(Z)=P_1$, and $\varrho_2(Z)=P_2$ .
Since the factorizations in \eqref{mostow} are unique, these maps are well defined. The product map $\tau(W, P_1, P_2)=W P_1 P_2$ is
 the inverse of $\varrho$. For any matrix $A$, let $\text{\rm cond} (A)$ denotes the \emph{condition number} of $A$.

\begin{thm}\label{3.1}
For $Z\in \mathbb{GL}(n,\C)$ let $\beta(Z)=\int_0^\infty \|e^{-t\left( (Z^* Z)^{-1} \overline{Z^* Z}\right)^{1/2}}\|^2 dt$. Then
\begin{equation}
|||{\De\varrho_0(Z)}|||\leq  \frac{\|P_1^{-1}\| \ \|P_2^{-1}\|}{2 } \left( 1+ \|P_1\|  \beta(Z)
\text{\rm cond} (Z)  \left(1+\text{\rm cond}(Z)^4\right)\right),\label{bound for W}
\end{equation}
\begin{equation}
|||{\De\varrho_1(Z)}|||\leq \frac{\text{\rm cond} (P_1) \ \|P_2^{-1}\| }{2 } \left(1+\|P_1\| \beta(Z)
\text{\rm cond} (Z)  \left(1+\text{\rm cond}(Z)^4\right)\right),\label{bound for P_1}
\end{equation}
and
\begin{eqnarray}
|||{\De \varrho_2(Z)}|||&\leq&  \beta(Z)
\text{\rm cond} (Z) \left(1+\text{\rm cond}(Z)^4\right). \label{bound for P_2}
\end{eqnarray}
\end{thm}
\begin{proof}

We know $\varrho_2(Z)=e^S=(Z^*Z\# \overline{Z^*Z})^{1/2}$.
Let $f:(0,\infty)\rightarrow \mathbb R$ be defined as $f (t)=t^{1/2}$, $g:\mathbb P(n,\mathbb C)\rightarrow \mathbb P(n,\mathbb C)\times \mathbb P(n,\mathbb C)$ as $g(A)=(A,\overline{A})$, and $h:GL(n,\mathbb C)\rightarrow \mathbb P(n,\mathbb C)$ as $h(Z)=Z^*Z$. Then $\varrho_2=f\circ G \circ g \circ h$, where $G$ is the geometric mean map as defined in Section 2. By the chain rule, $\De \varrho_2(Z)=\De f(Z^*Z\#\overline{Z^*Z})\circ \De G(Z^*Z, \overline{Z^*Z})\circ \De g(Z^*Z)\circ \De h(Z)$.

Now by \cite[Theorem X.3.1]{bhatia_matrix_analysis}, we obtain that if $A\in \mathbb P(n,\mathbb C)$, then
\begin{equation}
|||\De f(A)|||\leq \frac{1}{2}\|A^{-1}\|^{1/2}.\label{6.5}
\end{equation}
So
 $$|||\De \varrho_2(Z)(A)|||\leq \frac{1}{2}\|(Z^*Z\#\overline{Z^*Z})^{-1}\|^{1/2}\ |||\De G(Z^*Z, \overline{Z^*Z})(Z^*A+AZ, \overline{Z^*A+AZ})|||.$$
We know that $(A\#B)^{-1}=A^{-1}\# B^{-1}$ and $\|A\#B\|\leq \|A\|^{1/2} \|B\|^{1/2}$. Therefore
$$|||\De \varrho_2(Z)(A)|||\leq \frac{1}{2}\|Z^{-1}\|\ |||\De G(Z^*Z, \overline{Z^*Z})(Z^*A+AZ, \overline{Z^*A+AZ})|||.$$
Let $C=(Z^*Z)^{-1} \left(Z^*Z\#\overline{Z^*Z}\right)=\left((Z^*Z)^{-1}\overline{Z^*Z}\right)^{1/2}$. Then $\|C\|\leq \text{\rm cond} (Z)^2$. By \eqref{bound_gm}, we obtain
\begin{equation}
|||\De \varrho_2(Z)(A)|||\leq \frac{1}{2}\|Z^{-1}\| \beta(Z)
\left(1+\|C\|^2\right)|||Z^*A+AZ|||,\label{6.7}
\end{equation}
and so
\begin{equation}
|||\De \varrho_2(Z)(A)|||\leq  \beta(Z)\
\text{\rm cond} (Z)  \left(1+\text{\rm cond}(Z)^4\right)\ |||A|||.\label{4.7}
\end{equation}
Equation \eqref{bound for P_2} follows from \eqref{4.7}.

Let $\mathbb {SH}(n,\R)$ be the space of $n\times n$ real skew symmetric matrices. The tangent space at any point $P_1$ is given by $i P_1^{1/2} \mathbb {SH}(n,\R)P_1^{1/2}$. This follows from \cite[p. 258]{bhatia_Matrix_factorizations}.

Let
$ \De\varrho(Z): \mathbb{M}(n,\C)\to W\mathbb {SH}(n,\C)\oplus i P_1^{1/2} \mathbb {SH}(n,\R)P_1^{1/2}\oplus \mathbb H(n,\R)$ be defined as
$\De \varrho(Z)(A)=(W X, i P_1^{1/2} Y_1 P_1^{1/2}, Y_2)$, where $X\in \mathbb {SH}(n,\C)$, $Y_1\in \mathbb {SH}(n,\R)$ and $Y_2\in  \mathbb H(n,\R)$. So we have \begin{equation}X^*=-X,\ \overline{Y_1}=Y_1,\ Y_1^t=-Y_1,\ \overline{Y_2}=Y_2,\ Y_2^t=Y_2.\label{6.7}\end{equation}  The map $\De \varrho(Z)$ is the inverse of $\De \tau(W, P_1, P_2)$, and so
\begin{equation}
\De \varrho_0(Z)(A)=W X,\ \De \varrho_1(Z)(A)=i P_1^{1/2} Y_1 P_1^{1/2},\ \De \varrho_2(Z)(A)=Y_2,\label{eq3.9}
\end{equation}
and \begin{equation}
 \De \tau(W, P_1, P_2)(W X, i P^{1/2} Y_1 P^{1/2}, Y_2)=A.\label{6.8}
\end{equation}

Also,
\begin{eqnarray}
\De \tau(W, P_1,P_2)(W X, i P^{1/2} Y_1 P^{1/2}, Y_2)&=&\left.\frac{d}{dt}\right|_{t=0} \tau(We^{tX}, P_1^{1/2} e^{itY_1} P_1^{1/2}, P_2+t Y_2)\nonumber\\
&=&\left.\frac{d}{dt}\right|_{t=0} We^{tX} P_1^{1/2} e^{itY_1} P_1^{1/2} (P_2+t Y_2)\nonumber\\
&=& WXP_1P_2+W P_1^{1/2} (i Y_1)P_1^{1/2} P_2+WP_1 Y_2. \hspace{1.2cm}\label{eq3.10}
\end{eqnarray}
By \eqref{6.8} and \eqref{eq3.10}, we obtain
\begin{equation}
WXP_1P_2+W P_1^{1/2} (i Y_1)P_1^{1/2} P_2+WP_1 Y_2=A, \label{new}
\end{equation}
that is,
\begin{equation}
X+P_1^{1/2} (i Y_1) P_1^{-1/2}=(W^*A-P_1 Y_2)(P_1 P_2)^{-1}.\label{6.10}
\end{equation}
Taking conjugate transpose on both the sides and using \eqref{6.7}, we get
\begin{equation}
-X+P_1^{-1/2} (i Y_1) P_1^{1/2}=(P_1 P_2)^{-1}(A^*W-Y_2 P_1).\label{6.11}
\end{equation}
Adding \eqref{6.10} and \eqref{6.11} gives
$$(P_1^{1/2} (i Y_1)P_1^{1/2}) P_1^{-1}+ P_1^{-1} (P_1^{1/2} (i Y_1)P_1^{1/2})=\h \left((W^*A-P_1 Y_2)(P_1 P_2)^{-1}\right).$$
By \cite[Theorem VII.2.3]{bhatia_matrix_analysis}, we obtain
\begin{eqnarray}
P_1^{1/2} (i Y_1)P_1^{1/2}
&=& \int_0^{\infty} e^{-tP_1^{-1}} \h \left((W^*A-P_1 Y_2)(P_1 P_2)^{-1}\right)  e^{-tP_1^{-1}} dt.\label{6.13}
\end{eqnarray}
So
\begin{eqnarray*}
|||\De \varrho_1(Z)(A)|||&=& |||P_1^{1/2} (i Y_1)P_1^{1/2}|||\\
&\leq & \left(\int_0^{\infty} \|e^{-tP_1^{-1}}\|^2 dt\right)|||\h \left((W^*A-P_1 Y_2)(P_1 P_2)^{-1}\right)|||\\
&\leq& \frac{\|P_1\|}{2} |||(W^*A-P_1 Y_2)(P_1 P_2)^{-1}|||\\
&\leq& \frac{\text{\rm cond} (P_1) \ \|P_2^{-1}\| }{2 } \left(1+\|P_1\| \beta(Z)\
\text{\rm cond} (Z)  \left(1+\text{\rm cond}(Z)^4\right)\right)\ |||A|||.
\end{eqnarray*}
The last inequality follows from \eqref{4.7} and \eqref{eq3.9}. Hence we obtain \eqref{bound for P_1}.

By \eqref{new}, we also have
\begin{equation}
XP_1+P_1^{1/2} (i Y_1) P_1^{1/2}=(W^*A-P_1 Y_2) P_2^{-1}.\label{6.15}
\end{equation}
Again taking conjugate transpose on both the sides and using \eqref{6.7}, we obtain
\begin{equation}
P_1 X+P_1^{1/2} (i Y_1) P_1^{1/2}=P_2^{-1} (A^*W-Y_2 P_1).\label{6.16}
\end{equation}
Now, subtracting  \eqref{6.16} from \eqref{6.15}, we get
$$X P_1+P_1 X=2i \text{ Im}\left((W^*A-P_1 Y_2) P_2^{-1}\right) .$$
Again by \cite[Theorem VII.2.3]{bhatia_matrix_analysis}, we get \begin{eqnarray*}
X&=& \int_0^{\infty} e^{-t P_1} \text{ Im}\left((W^*A-P_1 Y_2) P_2^{-1}\right) e^{-t P_1} dt.
\end{eqnarray*}
Therefore \begin{eqnarray*}
|||\De\varrho_0(Z)(A)|||&=&||| X|||\\
&\leq&  \left(\int_0^{\infty}  \|e^{-t P_1}\|^2 dt\right) |||(W^*A-P_1 Y_2) P_2^{-1}|||\\
&\leq & \frac{\|P_1^{-1}\| \ \|P_2^{-1}\|}{2 } \left( 1+ \|P_1\|  \beta(Z)\
\text{\rm cond} (Z)  \left(1+\text{\rm cond}(Z)^4\right)\right) |||A|||.
\end{eqnarray*}
From this, \eqref{bound for W} follows.

\end{proof}
\begin{remark}
We have used in \eqref{4.7} that $\|A\# B\|\leq \|A\|^{1/2} \|B\|^{1/2}$.
Better bounds on $|||\De \varrho_2(Z)(X)|||$ can be found using \cite[Theorem 2]{bhatia_grover}. For example, we also have
\begin{eqnarray*}
|||\De \varrho_2(Z)|||&\leq& \|Z\|\ \beta (Z) \|(\overline{Z^*Z})^{-1/4}(Z^*Z)^{-1/2}(\overline{Z^*Z})^{-1/4}\|\\
&&\left(1+\|Z^{-1}\|^4 \|(\overline{Z^*Z})^{1/4}(Z^*Z)^{1/2}(\overline{Z^*Z})^{1/4}\|^2\right).
\end{eqnarray*}
\end{remark}
\begin{remark}
One can find another bound for  $|||\De \varrho_1(Z)|||$ in Theorem \ref{3.1} by using the expression $e^{iK}=e^{-S} Z^*Z e^{-S}$ given in \cite{bhatia_bipolar}. This can be expressed as $\varrho_1(Z)=\left(\varrho_2(Z)^{-1} (Z^*Z) \varrho_2(Z)^{-1}\right)^{1/2}.$ Using this approach, the factor $\frac{\text{\rm cond}(P_1)}{2}$ in \eqref{bound for P_1}  gets replaced by $\|P_1\|^2$. By the chain rule, we get
\begin{eqnarray*}
\De \varrho_1(Z)(A)&=& \De f(P_2^{-1} Z^*Z P_2^{-1})(2\h ( P_2^{-1} Z^*)(A P_2^{-1} -Z P_2^{-1} (\De (\varrho_2(Z)(A)) P_2^{-1})),
\end{eqnarray*}
where $f$ is the square root function.
By \cite[Theorem X.3.1]{bhatia_matrix_analysis} and using $ZP_2^{-1}=W P_1$, 
we obtain
\begin{eqnarray}
|||\De \varrho_1(Z)|||&\leq & \|P_1\|^2 \|P_2^{-1}\| \left(1+\|P_1\| \beta(Z)\ \text{\rm cond} (Z)\ (1+(\text{\rm cond} (Z))^4) \right).\hspace{0.2cm}\label{4.16}
\end{eqnarray}
\end{remark}

\section{Decomposition of unitary matrices}\label{section 4}
Every complex unitary matrix $W$ can be factorized as $W=W_1 W_2$, by the second or third polar decomposition of $W$.
This decomposition is unique if $W' W$ doesn't have $-1$ as an eigenvalue. Let $\mathbb U=\{W\in \mathbb U(n,\C) |  -1 \notin \sigma(W' W)\}$, where $\sigma(A)$ denotes the spectrum of $A$ and  $\mathbb U_{\text{sym}^+}$ be the set of $U\in \mathbb U(n,\C)$ such that $U'=U$ and $U$ has all the eigenvalues in the open right half plane. Let $\mathbb O(n,\R)$ be the set of real orthogonal matrices.
We define $\Phi: \mathbb U\rightarrow \mathbb O(n,\R) \times \mathbb U_{\text{sym}^+}$ as $\Phi(W)=(\Phi_1(W),\Phi_2(W))$, where $\Phi_1(W)=W_1$ and $\Phi_2(W)=W_2$. The product map $\Psi: \mathbb O(n,\R) \times \mathbb U_{\text{sym}^+}\rightarrow \mathbb U$ is the inverse of $\Phi$. 

\begin{thm}\label{thm:perturbation_bound_for unitary}
Let $\sigma(W_2)=\{e^{i\theta_1},\ldots,e^{i\theta_n}\}$. Let $\{a_n\}$ be any $\ell_1$-sequence such that for all $\theta= \theta_i-\theta_j \ (1\leq i,j\leq n)$ \begin{equation}\sum_{n=-\infty}^\infty (-1)^n a_n e^{in\theta}= \frac{1}{1+e^{i\theta}}.\label{an}\end{equation}
Then for $k=1,2$
\begin{equation}
|||{\De \Phi_k (W)}|||\leq 2 \left(\sum_{n=-\infty}^{\infty} |a_n|\right).
\end{equation}
\end{thm}

\begin{proof}
The map
$$\De \Phi(W): W\ \mathbb {SH}(n,\C)\to W_1\ \mathbb {SH}(n,\R)\oplus W_2^{\frac{1}{2}}\ i \mathbb {H}(n,\R)\ W_2^{\frac{1}{2}}$$ is an isomorphism and its inverse is $\De \Psi(W_1,W_2)$. For $X\in \mathbb {SH}(n,\R)$ and $Y\in i \mathbb {H}(n,\R)$

\begin{eqnarray*}
\De\Psi(W_1, W_2)(W_1 X, W_2^{\frac{1}{2}}Y W_2^{\frac{1}{2}})&= &\left.\frac{d}{dt}\right|_{t=0}\Psi\left(W_1e^{tX}, W_2^{\frac{1}{2}}e^{tY}W_2^{\frac{1}{2}}\right)\\
&=& W_1(XW_2+ W_2^{\frac{1}{2}}YW_2^{\frac{1}{2}}).
\end{eqnarray*}

Let $S\in \mathbb {SH}(n,\C)$ be such that  $\De \Phi(W)(WS)= (W_1 X, W_2^{\frac{1}{2}}Y W_2^{\frac{1}{2}})$. Then we have
\begin{equation*}
WS= W_1(XW_2+ W_2^{\frac{1}{2}}YW_2^{\frac{1}{2}}),\label{mainequation_in_unitary_decomposition}
\end{equation*}
that is, \begin{equation}
W_2 S W_2^{-1} = X+W_2^{1/2} Y W_2^{-1/2}.\label{*}
\end{equation}
Taking transpose on both the sides of the above equation \eqref{*} and adding  the new equation to \eqref{*}, we get
\begin{equation}
W_2Y+YW_2= W_2^{\frac{3}{2}}SW_2^{-\frac{1}{2}}+W_2^{-\frac{1}{2}}S'W_2^{\frac{3}{2}}.\label{Y_solution}
\end{equation}
By \cite[Theorem VII.2.7]{bhatia_matrix_analysis},  we obtain
\begin{eqnarray*}
Y&=&\sum_{n=-\infty}^\infty a_n(-1)^n\left(W_2^{-n+\frac{1}{2}}SW_2^{n-\frac{1}{2}}+W_2^{-n-\frac{3}{2}}S'W_2^{n+\frac{3}{2}}\right).\label{soln1}
\end{eqnarray*}
This gives $$|||{Y}|||\leq 2\left(\sum_{n=-\infty}^{\infty} |a_n|\right) |||{S}|||.$$ Therefore
\begin{eqnarray}|||{\De\Phi_2(W)(WS)}|||
&=&|||{Y}|||\nonumber\\
&\leq& 2\left(\sum_{n=-\infty}^{\infty} |a_n|\right) |||{WS}|||.
\label{dphi_2_unitary_decompstn}\end{eqnarray}

Equation \eqref{*} can also be written as
\begin{equation}
W_2^{1/2} S W_2^{-1/2}=W_2^{-1/2} X W_2^{1/2}+Y. \label{**}
\end{equation}
Taking complex conjugate on both sides of the above equation \eqref{**} and adding the new equation to \eqref{**}, we get
\begin{equation}
XW_2+W_2X= W_2S-S'W_2.
\end{equation}
By similar calculations as done above, we get $$|||{W_1 X}|||\leq 2 \left(\sum_{n=-\infty}^{\infty} {|a_n|}\right)  |||{S}|||,$$
and so
\begin{equation}|||{\De \Phi_1(W)}(WS)|||=|||{X}|||\leq 2\left(\sum_{n=-\infty}^{\infty} {|a_n|}\right) |||{WS}|||.\label{dphi_1_unitary_decompstn}\end{equation}
Equations \eqref{dphi_2_unitary_decompstn} and \eqref{dphi_1_unitary_decompstn} give the required result.
\end{proof}

\section{Perturbation bounds}

In this section, we discuss first order perturbation bounds for a map from  a Lie group to a manifold and use it to obtain perturbation bounds for the decomposition maps.  Let $\mathcal M\subseteq GL(n,\mathbb C)$ be a differentiable manifold. For $A\in \mathcal M$ let $\tilde A$ denote a perturbation of $A$ in a small neighbourhood of $A$ in $\mathcal M$. Suppose  $A=A_1 A_2$. Then  $\tilde{A_i}$ denote the corresponding factors for $\tilde A$. Let $f$ be a smooth function on $\mathcal M$. If $\mathcal M$ is a convex set, then by Taylor's theorem, we have
\begin{equation}\label{taylorthmforlinearspace}||f(\tilde A)-f(A)|||\leq |||\De f(A)|||\ |||\tilde A-A|||+O(|||\tilde A-A|||^2).\end{equation} We denote this as $$|||f(\tilde A)-f(A)|||\lesssim |||\De f(A)|||\ |||\tilde A-A|||.$$
We also note here that if there is a $M>0$ such that $|||Df(A)|||<M$, then in a small neighborhood of  $A$,  we have
    \begin{equation}|||f(\tilde{A})- f(A)|||<M|||\tilde{A}- A|||\label{realbound}.\end{equation}
\subsection{First order perturbation bounds}
The function $\log$ is well defined for all non singular matrices $A$ if we choose a branch of  logarithm.  In this case,  $\exp$ is its inverse.
 The map $\De \log (A):\mathbb M(n,\mathbb C)\to \mathbb M(n,\mathbb C)$ is given by
\begin{equation}\label{derivativeoflog}\De \log (A)(X)= \int_0^1\left(t(A-I)+I\right)^{-1}X\left(t(A-I)+I\right)^{-1} dt.\end{equation}

For $\epsilon >0$
define $U_\epsilon=\{X\in \mathbb M_n(\mathbb C): \|X\|_2<\epsilon\}$ and $V_\epsilon= \exp(U_\epsilon)$.
Let $G\subseteq GL(n,\mathbb C)$ be a matrix Lie group with Lie algebra $\mathcal{G}$ and let $A_0\in G$. Then by \cite[Theorem 2.27]{hall}, there exists an $\epsilon>0$
such that the map $H: U_\epsilon \cap \mathcal{G}\to A_0V_\epsilon\cap G$ defined by  $H(X)= A_0 \exp(X)$ is a bijective map. For $X\in \mathcal{G}$
\begin{eqnarray*}\De H(O)(X)&=&A_0  \De\exp(O)(X)\\
&=& A_0\int_0^1 e^{(1-t)O}Xe^{tO}dt\\
&=&A_0 X.
\end{eqnarray*}
This gives
\begin{equation}\label{DpsiX}
|||\De H(O)|||\leq \|A_0\| .
\end{equation} Let $H_1: A_0V_\epsilon\to M(n, \mathbb C)$ be the map defined as $H_1(W)=\log (A_0^{-1}W).$ Note that the restriction of $H_1$ to $A_0V_\epsilon\cap G$ is $H^{-1}$, and so $(H, U_\epsilon\cap \mathcal{G}, A_0V_\epsilon \cap G)$ is a local chart around $A_0\in G$.  For $A\in A_0V_\epsilon$
\begin{eqnarray*}
\De H_1(A)(X)&=& \De \log{(A_0^{-1}A)}(A_0^{-1}X)\\
&=& \int_0^1\left(t(A_0^{-1}A-I)+I\right)^{-1}A_0^{-1}X\left(t(A_0^{-1}A-I)+I\right)^{-1} dt\\
&=& \int_0^1\left(t(A-A_0)+A_0\right)^{-1}X\left(t(A-A_0)+A_0\right)^{-1}A_0\ dt.
\end{eqnarray*}

By Taylor's theorem, we have
\begin{eqnarray*}
|||H_1(\tilde{A})-H_1(A)|||&\lesssim& |||\De H_1(A)|||\ |||\tilde{A}-A||| \\
&\leq&\|A_0\|\left(\int_0^1\|\left(t(A-A_0)+A_0\right)^{-1}\|^2dt\right)  |||\tilde{A}-A|||.
\end{eqnarray*}

In particular, when $A= A_0$,
we have
\begin{eqnarray}\label{derivatieofpsi_inverse}
|||H_1(\tilde{A_0})-H_1(A_0)|||\lesssim \|A_0\|\ \|A_0^{-1}\|^2 \ |||\tilde{A_0}-A_0|||.
\end{eqnarray}

Let $G_1\subseteq GL(n, \mathbb C) $ be a differential manifold ($G_1$ may not be a group) and let $F: G\to G_1$ be a smooth map.  Then
$F\circ H: U_\epsilon\cap \mathcal{G}\to G_1.$
Note that $ H_1(A_0)=O$. Let $\tilde{S}\in U_\epsilon\cap \mathcal{G}$  be in a small neighbourhood of $O$. Then we have
\begin{equation}
|||(F\circ H)(\tilde{S})-(F\circ H)(O)|||\lesssim|||\De(F\circ H)(O) |||\ |||\tilde{S}|||.
\end{equation}
Therefore
\begin{equation}
|||(F\circ H)(\tilde{S})-(F\circ H)(O)|||\lesssim|||\De F(A_0)|||\ |||\De H(O)|||\ |||\tilde{S}|||\label{Dphi}.
\end{equation}

Let $\tilde{A_0}=H(\tilde S)$.
Using equations \eqref{DpsiX} and \eqref{Dphi}, we get
\begin{eqnarray*}
|||F(\tilde{A_0})-F(A_0)|||&\lesssim& |||\De F(A_0)||| \ \|A_0\|\ |||\tilde{S}|||\\
&=& |||\De F(A_0)|||\ \|A_0\|\ |||H_1(\tilde{A_0})-H_1(A_0)|||.
\end{eqnarray*}
By \eqref{derivatieofpsi_inverse}, we obtain
\begin{eqnarray}
|||F(\tilde{A_0})-F(A_0)|||\lesssim |||\De F(A_0)|||\ \text{\rm cond} (A_0) ^2 \ |||\tilde{A_0}-A_0|||.
\end{eqnarray}

In particular, when $A_0$ is unitary matrix, we get
\begin{equation}
|||F(\tilde{A_0})-F(A_0)|||\lesssim |||\De F(A_0)|||\ |||\tilde{A_0}-A_0|||. \label{6.6}
\end{equation}

\subsection{Perturbation bounds for the bipolar decomposition}
Equation \eqref{6.6} and Theorem \ref{thm:perturbation_bound_for unitary} together give the perturbation bounds for the decomposition  \eqref{unitary}. We state this as a proposition below. The notations are as in Section \ref{section 4}.
\begin{prop}
For $W\in \mathbb U$ and $k=1,2$
\begin{equation}
|||\tilde{W_k}-W_k|||\lesssim 2\left(\sum_{n=-\infty}^\infty |a_n|\right) |||\tilde{W}-W|||.\label{bound expression_unitary_decompstn}
\end{equation}
\end{prop}

As observed in \cite{bhatia_bipolar}, the expression \eqref{unitary} gives both the second and third polar decompositions for a unitary matrix $W$.  Therefore Theorem 4.6 and Theorem 4.7 in \cite{bhatia_Matrix_factorizations} give that for each $k$\begin{equation}|||\tilde{W}_k-W_k|||\lesssim\left(\displaystyle{\int_0^\infty}\|{e^{-tW_2}}\|^2 dt\right) |||\tilde{W}-W|||.\label{bound_from Second and third} \end{equation}

We see that the bounds obtained in \eqref{bound expression_unitary_decompstn} are sometimes better than the ones given  by \eqref{bound_from Second and third}. For example,
let   $W= \diag(e^{i\theta}, e^{i\theta})$, where $\pi/3<\theta<\pi/2$.  Then $W_1=I$, $W_2=W$ and$\displaystyle{\int_0^\infty}\|{e^{-tW_2}}\|^2 dt= \frac{1}{2\cos\theta}>1$.
Let $a_0=1/2$ and $a_n=0$ for all $n\neq 0$. Then $\{a_n\}\in \ell_1$ satisfies \eqref{an} and $2\sum_{n=-\infty}^{\infty} |a_n|=1$.

If the eigenvalues of $W_2$ are close to $i$ or $ -i$, then the bounds in \eqref{bound_from Second and third} are too large. But the bounds we get in \eqref{bound expression_unitary_decompstn} depend upon how far the eigenvalues of $W_2$ lie on the unit circle. We explain this below.

Let $\Theta= \{\theta_i-\theta_j: e^{i\theta_j}\in \sigma(W_2)\}\subseteq (-\delta, \delta),$ where $0<\delta<\pi$.  We define the function $f: [-\pi, \pi]\to \mathbb C$ as
\[
  f(\theta) =
  \begin{cases}
    \frac{1}{2}+ \frac{i\tan\frac{\delta}{2}}{2(\pi-\delta)}(\theta+\pi) & \text{$-\pi\leq \theta \leq -\delta$,} \\
    \frac{1}{2}-i\tan\frac{\theta}{2}& \text{ $-\delta\leq \theta\leq \delta$,}\\
    \frac{1}{2}+\frac{i}{2}\tan\frac{\delta}{2}(-1+\frac{\theta-\delta}{(\pi-\delta)})&\text{ $\ \delta\leq \theta\leq \pi$.}
  \end{cases}
\]
Then $f$ is periodic and absolutely continuous. Also, $$\int_{-\pi}^{\pi}|f'(\theta)|^2d\theta = \frac{\tan^2\frac{\delta}{2}}{2(\pi- \delta)}+\frac{\tan^2\frac{\delta}{2}}{12}+\frac{\tan\frac{\delta}{2}}{4}.$$ So $f'\in L^2[-\pi, \pi]$.
Let the Fourier coefficients of $f$ be $b_n$.  For the sequence $a_n=(-1)^n b_n$ we have $a_n\in \ell^1$ and for $\theta\in (-\delta, \delta)$
 \begin{equation}
\sum_{n=-\infty}^\infty (-1)^n a_n e^{in\theta}= \frac{1}{1+e^{i\theta}}= \frac{1}{2}-\frac{i}{2}\tan\frac{\theta}{2}.
\end{equation}
From \cite[p. 117]{bhatia_fs}, we know that $$\sum_{n=-\infty}^{\infty}|b_n|= \sum_{n=-\infty}^{\infty}|a_n| \leq a_0+ \frac{\pi}{\sqrt{3}}\|f'\|_{L^2}.$$
Therefore \begin{equation}\label{boundan}2\sum_{n=-\infty}^{\infty} |a_n|\leq 1+ \frac{\pi}{\sqrt{3}}\sqrt{\frac{2\tan^2\frac{\delta}{2}}{(\pi- \delta)}+\frac{\tan^2\frac{\delta}{2}}{3}+\tan\frac{\delta}{2}}.\end{equation}
If $\delta$ is very small, that is, if the eigenvalues of $W_2$ are very close to each other, then by \eqref{boundan}, we see that  $2\sum a_n$ is very close to 1.

We now obtain the perturbation bounds for the factors $S,K, L, T$ in the bipolar decomposition \eqref{bipolar}.
 For $Z\in GL(n,\mathbb C)$ let $$k(Z)=\left(\int_0^\infty \|e^{-t\left( (Z^* Z)^{-1} \overline{Z^* Z}\right)^{1/2}}\|^2 dt\right)\
\text{\rm cond} (Z)  \left(1+\text{\rm cond}(Z)^4\right),$$ and for $W$ unitary let $$C(W)=\int_0^1\|(t(W-I)+I)^{-1}\|^2dt.$$

Before stating the theorem, we observe a few things about the decomposition \eqref{bipolar}. For any $Z\in GL(n,\mathbb C)$ the matrices $S$ and  $K$ are unique but $L$ and $T$ are not unique.
If $e^L$ and $e^{iT}$ do not have $-1$ as an eigenvalue, then we can use principal logarithm to define $L$ and $T$ uniquely.  But if $e^L$ or $e^{iT}$ have $-1$ as an eigenvalue, then we choose  $\alpha\in [-\pi, 0)$ such that $e^{i\alpha}\notin \sigma(W_1)\cup \sigma(W_2)$.  A branch of logarithm for which $\arg z\in [\alpha, \alpha+2\pi)$ gives unique $S$, $K$, $L$ and $T$.

\begin{thm}\label{main}
Let $Z\in GL(n,\mathbb C)$ be such that $-1\notin \sigma(Z'Z ((Z^*Z)^{-1}\# (\overline{Z^*Z})^{-1}))$. Let $Z=e^L e^{iT} e^{iK} e^S$, where $\sigma(e^{iT})=\{e^{i\theta_1},\ldots,e^{i\theta_n}\}$. Let $\{a_n\}$ be any $\ell_1$-sequence such that for all $\theta= \theta_i-\theta_j \ (1\leq i,j\leq n)$ \begin{equation*}\sum_{n=-\infty}^\infty (-1)^n a_n e^{in\theta}= \frac{1}{1+e^{i\theta}}.\end{equation*}Then
\begin{equation}
|||\tilde L-L|||\lesssim 2\ C({e^L})\left(\sum_{n=-\infty}^{\infty} |a_n|\right)\frac{\|e^{-iK}\| \ \|e^{-S}\|}{2 } \left( 1+ \|e^{i K}\|\ k (Z)\right) |||\tilde Z-Z|||,\label{bound for L}
\end{equation}
\begin{equation}
|||\tilde{T}-T|||\lesssim 2\ C({e^{iT}})\left(\sum_{n=-\infty}^{\infty} |a_n|\right)\frac{\|e^{-iK}\| \ \|e^{-S}\|}{2 } \left( 1+ \|e^{i K}\|\ k (Z)\right) |||\tilde Z-Z|||,\label{bound for T}
\end{equation}
\begin{equation}
|||\tilde{K}-K|||\lesssim|||e^{-iK}|||\ \frac{\text{\rm cond} (e^{iK}) \ \|e^{-S}\| }{2 } \left(1+\|e^{iK}\|\ k(Z) \right) |||\tilde Z-Z|||,\label{bound for K}
\end{equation}
and
\begin{equation}
|||\tilde{S}-S|||\lesssim|||e^{-S}|||\ k(Z)\ |||\tilde Z-Z||| \label{bound for S}.
\end{equation}
\end{thm}

\begin{proof}
Let $Z= W e^{iK}e^S$. Using notations as in section 3, we get
\begin{equation}
|||\tilde{W}-W|||\lesssim |||\De\varrho_0(Z)|||\ |||\tilde{Z}-Z|||,
\end{equation}
\begin{equation}
|||e^{i\tilde{K}}- e^{iK}|||\lesssim |||\De\varrho_1(Z)|||\ |||\tilde{Z}-Z|||,
\end{equation}
and \begin{equation}
|||e^{\tilde{S}}- e^S|||\lesssim |||\De \varrho_2(Z)|||\ |||\tilde{Z}-Z|||.
\end{equation}

The matrices $e^S$ and $e^{iK}$ are both positive.  So  $\log(e^S)=S$ and $\log(e^{iK})= iK$.
We have
\begin{eqnarray}
|||\tilde S-S|||&\lesssim& |||\De \log (e^S)|||\ |||e^{\tilde S}-e^S|||\nonumber\\
&\lesssim & |||\De \log (e^S)|||\ |||\De \varrho_2(Z)|||\ |||\tilde Z-Z|||.
\end{eqnarray}
Similarly,
\begin{equation}
|||\tilde K-K|||\lesssim |||\De \log(e^{iK})|||\ |||\De \varrho_1(Z)||| \ |||\tilde Z-Z|||.
\end{equation}
We know that if $\eta$ is an operator monotone function on $(0,\infty)$, then for $A\in \mathbb P(n,\mathbb C)$
$|||\De \eta(A)|||\leq \|\eta'(A)\|$ \cite[Theorem X.3.4]{bhatia_matrix_analysis}.  Now since $\log$ is an operator monotone function on $(0,\infty)$, we obtain \begin{equation}
|||\De \log(e^S)|||\leq |||e^{-S}|||\label{5.22}
\end{equation} and
\begin{equation}
|||\De \log(e^{iK})|||\leq |||e^{-iK}|||.\label{5.23}
\end{equation}
Equations \eqref{5.22} and \eqref{bound for P_2} give \eqref{bound for S}. And, \eqref{5.23} and \eqref{bound for P_1} give \eqref{bound for K}.

 Since $-1\notin \sigma(Z'Z ((Z^*Z)^{-1}\# (\overline{Z^*Z})^{-1}))$, $-1\notin \sigma(W'W)$. Then by the second or third polar decomposition, $W$ can be uniquely factorized as $W=W_1W_2$, where $W_1$ and $W_2$ are also unitary matrices. If $L= \log W_1$ and $iT= \log W_2$, then we have $W= e^L e^{iT}$. Now
\begin{eqnarray*}
|||\tilde{L}-L|||&=& |||\log (e^{\tilde{L}})-\log (e^L)|||\\
&\lesssim& |||\De\log(e^L)|||\ |||e^{\tilde{L}}-e^L|||.
\end{eqnarray*}
By \eqref{derivativeoflog}, we obtain
$$
|||\tilde{L}-L|||\lesssim  C({e^L}) \ |||e^{\tilde{L}}-e^L|||.$$
Equation \eqref{bound expression_unitary_decompstn} gives
$$|||\tilde{L}-L|||\lesssim 2 C({e^L})\left(\sum_{n=-\infty}^{\infty} |a_n|\right)|||\tilde{W}-W|||.$$
Therefore we have
$$|||\tilde{L}-L|||\lesssim 2 C({e^L})\left(\sum_{n=-\infty}^{\infty} |a_n|\right)\ |||\De \varrho_0(Z)|||\ |||\tilde{Z}-Z|||.$$
By \eqref{bound for W}, we obtain \eqref{bound for L}.
Similar calculations by putting   $\tilde T=\frac{1}{i} \log \tilde W_2$ yield \eqref{bound for T}.
\end{proof}
We illustrate the behaviour of the above bounds with the help of an example. For a natural number $n$, consider a one-parameter family of matrices $Z_n(t)= \diag(e^{\sin t}, e^{\sin (t+\frac{\pi}{n})})$, $t\in \mathbb R$.  For $Z_n(t)$, the factors in the bipolar decomposition are given by $S_n(t)= \diag(\sin t, \sin(t+\frac{\pi}{n}))$, $K_n(t)=T_n(t)= L_n(t)= O$. We consider the operator norm in Theorem \ref{main}. Let $f_n(t)$ be the first order perturbation bounds as given in \eqref{bound for S}, that is, $f_n(t)= \|e^{-S_n(t)} \| \ k(Z_n(t))$. Then  \begin{eqnarray}f_n(t)&=& \frac{1}{2}\left(\max(e^{-\sin t}, e^{-\sin(t+\frac{\pi}{n})})\right)^2 \max(e^{\sin t}, e^{\sin(t+\frac{\pi}{n})})\nonumber\\ &&\left(1+\max(e^{-4\sin t}, e^{-4\sin(t+\frac{\pi}{n})}) \max(e^{4\sin t}, e^{4\sin(t+\frac{\pi}{n})})\right).\nonumber
    \end{eqnarray}
The behavior of $f_n(t)$ can be seen in the following graph.  We observe that for $n=2$, the perturbation bound for some of these matrices can be more than 1200.
  \begin{center}\includegraphics[width=\textwidth, height=6cm]{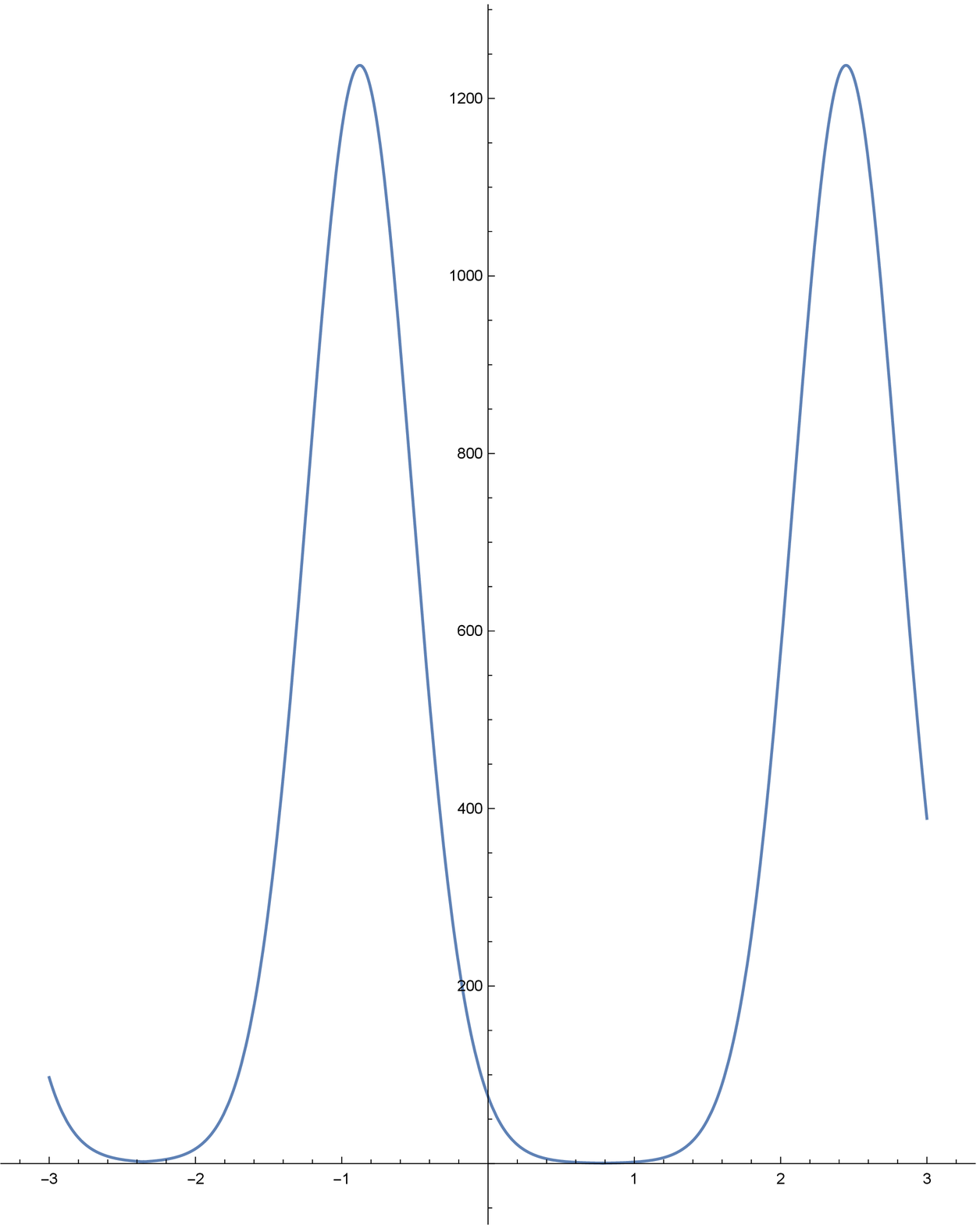}\end{center}
When $n$ increases, the maximum value of $f_n(t)$ decreases. In particular, we observe this for $n=500$ in the below graph.
\begin{center}\includegraphics[width=\linewidth]{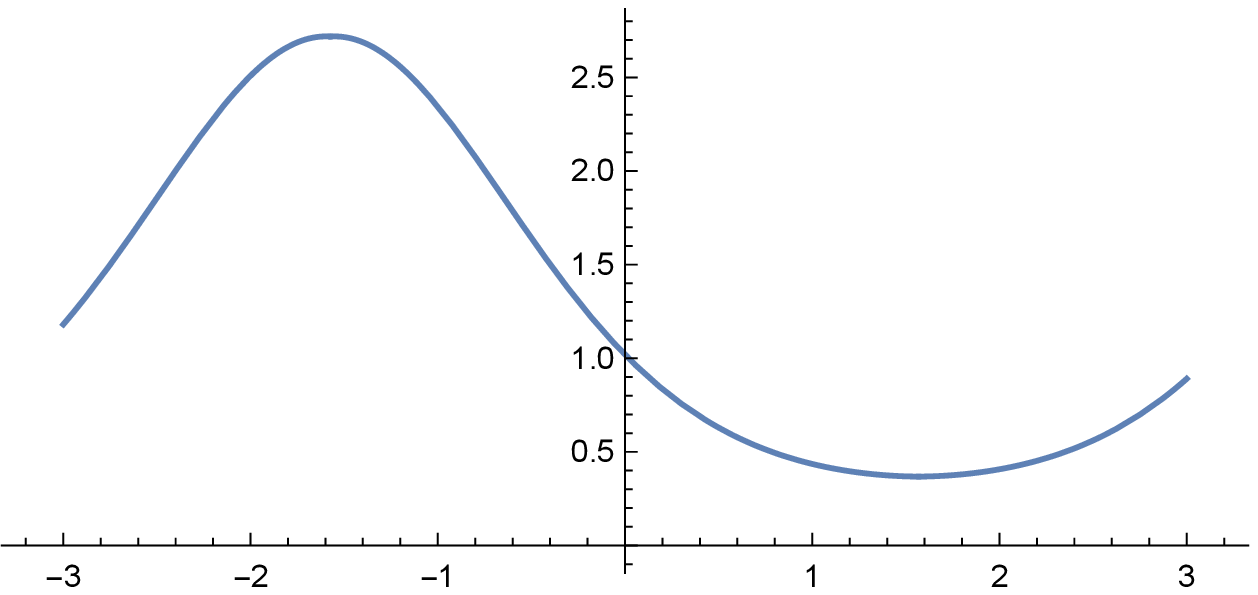}\end{center}
Bounds for other factors $K_n(t), L_n(t)$ and $T_n(t)$ are given by $g_n(t):= \frac{1}{2}\|e^{-S_n(t)}\| (1+k(Z_n(t)))$ which also vary in a similar way.

\begin{remark} Other perturbation bounds for $L$ and $T$ in Theorem \ref{main} can also be found using direct formulas, which we get from the principal logarithm. Let $\mathbb V$ be the set of  complex unitary matrices $W$ such that $W'W$ and $W e^{-\frac{1}{2}\log(W'W)}$ do not have eigenvalue $-1$. Then $L=\log(W e^{-\frac{1}{2}\log(W'W)})$ and $ T=  \frac{1}{2i}\log(W'W).$
Using the chain rule and Taylor's theorem, we get
$$|||\tilde{L}-L|||\leq \left(1+C(W'W)\right)|||\tilde{W}-W|||$$
and
$$|||\tilde{T}-T|||\leq C(W'W)|||\tilde{W}-W|||.$$
\end{remark}

\textbf{Acknowledgement.} The work of the first author is supported by  the research grant of INSPIRE Faculty Award [DST/INSPIRE/04/2014/002705] of Department of Science and Technology, India.

\end{document}